\documentclass[submission,copyright,creativecommons]{eptcs}
\providecommand{\event}{AiML 2026} 

\usepackage{aiml26}

\usepackage{iftex}

\ifpdf
  \usepackage{underscore}         
  \usepackage[T1]{fontenc}        
\else
  \usepackage{breakurl}           
\fi

\usepackage{amsmath,amsthm}
\usepackage{amssymb}
\usepackage{comment}
\usepackage{amsmath,amsthm,amsfonts,amssymb,bm,mathtools,mathrsfs}
\usepackage{enumitem}
\usepackage{hyperref}

\usepackage{multirow}

\usepackage{tikz,tikz-cd}
\usepackage{quiver}

\newcommand{\sset}{\subseteq}
\newcommand{\rset}{\supseteq}

\renewcommand{\phi}{\varphi}

\title{Conditionals and Modalities in Constructive Quantum Logics}
\author{Juan P. Aguilera
\institute{TU Wien, Austria}
\email{aguilera@logic.at}
\and
Guillaume Massas
\institute{Chapman University
California, USA}
\email{massas@chapman.edu}
}

\newcommand{\titlerunning}{Conditionals and Modalities in Constructive Quantum Logics}
\newcommand{\authorrunning}{J.P. Aguilera \& G. Massas}

\hypersetup{
  bookmarksnumbered,
  pdftitle    = {\titlerunning},
  pdfauthor   = {\authorrunning},
  pdfsubject  = {constructive quantum logic},               
  pdfkeywords = {Intuitionistic Logic, Quantum Logic, Fundamental Logic} 
}

\begin{document}
\maketitle
\begin{abstract}
We investigate logics that generalize both intuitionistic logic and quantum logic. In earlier work, we introduced Ex-logic, an extension of Holliday's fundamental logic that coincides with the intersection of orthologic and the implication-free fragment of intuitionistic logic. In this paper, we add an implication connective to Ex-logic and axiomatize iEx-logic, the intersection of full intuitionistic logic and orthomodular logic with the implication connective interpreted as the Sasaki hook $a\to b:= \lnot a \vee (a\wedge b)$. As a consequence, we obtain a characterization of the lattice of logics extending iEx-logic as the product of the lattice of intermediate logics and the lattice of orthomodular logics. We also explore the robustness of our algebraic approach by briefly discussing extensions of iEx-logic with modal operators.
\end{abstract}

\section{Introduction}
Recently introduced by Holliday \cite{Ho23}, fundamental logic is a propositional logic that generalizes both constructive reasoning and quantum reasoning. However, the intersection of the consequence relations on the set of formulas in the signature $\{\land,\lor,\neg\}$ determined by intuitionistic logic and orthologic respectively is a proper superset of the consequence relation determined by fundamental logic. Answering a question of Holliday, it was shown in \cite{AgMa26} that the intersection of intuitionistic logic and orthologic in the signature $\{\wedge,\vee,\lnot\}$ is axiomatized over fundamental logic by the following axiom
\textup{(Ex)}:
\begin{align*}
    \lnot \bigg[a \wedge \Big((b\wedge c) \vee (b &\wedge d)\Big)\bigg] \wedge a \wedge (c \vee e) \wedge \lnot\lnot f \vdash 
\lnot\lnot(a \wedge f) \\
    & \wedge \bigg[(a\wedge c) \vee (a \wedge e) \vee f\bigg] \wedge 
    \bigg[\Big( b \wedge (c \vee d)\Big) \vee \lnot \Big( b \wedge (c \vee d) \Big)\bigg].
\end{align*}

This result motivates the study of constructive quantum logics, i.e., propositional logics extending fundamental logic + \textup{(Ex)}. Here, \textit{orthologic} or \textit{minimal quantum logic} is the logic of ortholattices and results from classical logic by removing the axiom of distributivity. The main reason for considering the signature $\{\wedge, \vee, \lnot\}$ is that orthologic lacks a reasonable notion of implication (e.g., one satisfying \textit{modus ponens}).

However, some quantum logics extending orthologic do admit reasonable forms of implication.
As shown by Dalla Chiara and Giuntini \cite{DCGi02}, the so-called \textit{Sasaki implication} defined in an ortholattice $O$ by \[a \to b := \lnot a \vee (a\wedge b)\] satisfies \textit{modus ponens} if and only if $O$ is orthomodular. The orthomodular law \[\text{ if $a \leq b$, then $a \vee (\lnot a \wedge b) = b$}\] was already considered by Birkhoff and Von Neumann \cite{BvN36} as a principle for quantum logic. As implication is arguably a crucial component of intuitionistic logic, this motivates the question: \textit{What are the jointly valid principles of intuitionistic logic and orthomodular logic with the Sasaki implication?} Our main goal here is to answer this question. Note that, when speaking of implication in the context of orthomodular logic below, we shall always mean the Sasaki implication.

In \cite{Ho23p}, Holliday considers binary operations defined on bounded lattices that he calls preconditionals. Preconditionals are meant to provide a minimal theory of conditionals. Holliday shows that, whenever $L$ is a bounded lattice equipped with a preconditional $\to$ that satifies \textit{modus ponens}, $L$ is an orthomodular lattice if and only if $\to$ satisfies double negation elimination (i.e., the inequality $(a \to 0) \to 0 \leq a$), while $L$ is a Heyting algebra if and only if $\to$ satisfies the \textit{verum ex quodlibet} principle (i.e., the inequality $a \leq b \to a$). However, it is easy to see that the logic of preconditionals that satisfy \textit{modus ponens} is weaker than the intersection of orthomodular logic and intuitionistic logic. In order to axiomatize the latter, we  work in the signature $\{\wedge, \vee, \to\}$, recasting the negation axioms of fundamental logic in terms of implication via the definition $\lnot a := a \to 0$, and we generalize axiom $\mathrm{(Ex)}$ to a new axiom $\mathrm{(iEx)}$ (for ``\textit{implicative $\mathrm{(Ex)}$}'').
Our main theorem is the following:

\begin{theorem}\label{TheoremIntroEx}
The joint validities of intuitionistic logic and orthomodular logic in the full signature $\{\vee,\wedge,\to\}$ are axiomatized by the following axioms over fundamental logic:
\begin{enumerate}
\item $\mathrm{(iEx)}$;
\item $\lnot (a \to b) = \lnot(\lnot a \vee (a\wedge b))$;
\item $a \wedge (a \to b) = a \wedge b$;
\item $(a \wedge b) \to b = 1$;
\item $a \to (b \land c) \leq a \to b$.
\end{enumerate}
\end{theorem}

One might argue that the last three of these axioms merely capture the intended meaning of an implication satisfying \textit{modus ponens}, and thus that the ``content'' of the logic is captured by the first two axioms. We shall also give alternate axiomatizations by decomposing $\mathrm{(iEx)}$ into three axioms $\mathrm{(Nu)}$, $\mathrm{(Vi)}$, and $\mathrm{(iCl)}$, the first two of which only involve the connectives $\{\wedge,\vee,\lnot\}$ and are already valid in Ex-logic.\\

We proceed as follows. Section \ref{iexsec} starts by introducing a large class of algebras, called \textit{fundamental algebras}, that essentially correspond to fundamental lattices in a more general signature and generalize Heyting algebras and orthomodular lattices equipped with the Sasaki implication. We focus on a particular class of fundamental algebras, which we call \textit{Ex-algebras}, and introduce the axiom \textup{(iEx)}. In Section \ref{Sectmainthm}, we prove our main technical result, the iEx-Embedding Theorem, which establishes that any Ex-algebra subdirectly embeds into the product of an orthomodular algebra and a Heyting algebra. This result is used in Section \ref{icqsec} to show that iEx-logic, the logic of Ex-algebras, coincides with the intersection of intuitionistic logic and orthomodular logic and to characterize the lattice of extensions of iEx-logic. Finally, Section \ref{modalsec} briefly discusses adding modal operators to fundamental algebras.
\section{Ex-algebras} \label{iexsec}

In this section, we introduce the algebraic structures that we will be working with. The most general notion is that of a fundamental algebra, which subsumes Heyting algebras and orthomodular lattices.
Below, we shall always write $\lnot a$ as shorthand for $a \to 0$.  

\subsection{Fundamental algebras} \label{SectFL}

\begin{definition}
A structure $(L, \wedge, \vee, 0, 1, \to)$ is said to be a \textit{fundamental algebra} if $(L, \wedge, \vee, 0, 1)$ is a bounded lattice and, letting $\lnot a := a\to 0$,  $\lnot$ is a \textit{weak pseudo-complement} i.e., a unary operation for which the following hold:
\begin{enumerate}
\item \textit{antitonicity}, i.e., $a \leq b$ implies $\lnot b \leq \lnot a$ for all $a,b \in L$;
\item \textit{semi-complementation}, i.e., $a \wedge \lnot a = 0$ for all $a \in L$;
\item \textit{double-negation introduction}, i.e., $a \leq \lnot\lnot a$ for all $a \in L$.
\end{enumerate}
\end{definition} 

The $\{\land,\lor,\neg\}$-reduct of any fundamental algebra is a fundamental lattice in the sense of \cite{AgMa26}, also called a weakly pseudo-complemented lattice by Holliday in \cite{Ho23}. We recall that, in any fundamental lattice $L$, $\neg$ is dually self-adjoint, which means in particular that $\neg a = \neg \neg \neg a$ for any $a \in L$. Since we only impose requirements on the implication $\to$ when the second component is $0$, fundamental algebras generalize Holliday's class of preconditionals satisfying \textit{modus ponens} \cite{Ho23p}.

Given a weakly pseudo-complemented, bounded lattice $L$, a \textit{valuation} $v$ into $L$ assigns elements of $L$ to propositional variables. Valuations can be extended to arbitrary propositional formulas in the natural way. 
Often we abuse notation by omitting mention of $v$ if clear from the context, and identifying formulas with elements of $L$.
We say that $\varphi\vdash \psi$ is \textit{valid} in $L$, if $\varphi \leq \psi$ holds for all valuations into $L$.

\textit{Fundamental logic} can be defined as the logic of the $\{\land,\lor,\neg\}$-reducts of fundamental algebras. Alternatively, it can be defined as the weakest \textit{introduction--elimination} logic, i.e., the weakest logic such that for all formulas $\phi,\psi, \chi$, the following are valid:
\vspace{1em}

\begin{minipage}{0.48\textwidth}
\begin{enumerate}
\item $\phi\vdash\phi$;
\item $\bot \vdash \phi$;
\item $\phi \vdash \top$;
\item $\phi\wedge\psi \vdash \phi$;
\item $\phi\wedge\psi\vdash \psi$;
\item $\phi\vdash \phi\vee\psi$;
\item $\psi\vdash \phi\vee\psi$;
\end{enumerate}
\end{minipage}
\hfill
\begin{minipage}{0.48\textwidth}
\begin{enumerate}[start=8]
\item $\phi\vdash\lnot\lnot\phi$;
\item $\phi\wedge\lnot\phi\vdash\bot$;
\item if $\phi\vdash\psi$ and $\psi\vdash\chi$, then $\phi\vdash\chi$;
\item if $\phi\vdash \psi$ and $\phi\vdash\chi$, then $\phi\vdash\psi\wedge\chi$; 
\item if $\phi\vdash\chi$ and $\psi\vdash \chi$, then $\phi\vee\psi\vdash \chi$;
\item if $\phi\vdash\psi$, then $\lnot\psi\vdash\lnot\phi$.
\end{enumerate}
\end{minipage}
\vspace{1em}

\begin{definition}
An \textit{orthomodular algebra} is a fundamental algebra \mbox{$(L, \wedge, \vee, 0, 1, \lnot, \to)$} in which $\lnot$ satisfies double-negation elimination $\lnot\lnot a = a$ for all $a \in L$ and $\to$ satisfies \textit{modus ponens} $a \wedge (a\to b) 
\leq b$ and the Sasaki axiom $a \to b = \lnot a \vee (a \wedge b)$.
\end{definition}

It is well known that orthomodular logic with the Sasaki implication is the logic of orthomodular algebras; see Dalla Chiara and Giuntini \cite{DCGi02}. 
For intuitionistic logic, we must impose a strong interaction condition between $\land$ and $\to$ :

\begin{definition}
A fundamental algebra $(L, \wedge, \vee, 0, 1, \to)$ is a \textit{Heyting algebra} if it satisfies the following Residuation Law: for any $a,b,c\in L$, we have 
\[(c\wedge a)\leq b \text{ if and only if } c\leq a\to b.\]
\end{definition}
Recall that a lattice is \textit{distributive} if meets distribute over joins, and joins distribute over meets. A standard argument shows that distributivity follows from the Residuation Law, so that a Heyting algebra is always a distributive lattice. As is also well known, intuitionistic logic is the logic of Heyting algebras. We also recall the following result, originally due to Glivenko \cite{GL29}:

\begin{theorem}[Glivenko's Theorem]
    For any propositional formulas $\phi$ and $\psi$, if $\phi \leq \psi$ is classically valid, then $\phi \leq \neg \neg \psi$ is intuitionistically valid.
\end{theorem}

We now introduce a particular class of fundamental algebras that we call \textit{Ex-algebras}. As we shall see, Ex-algebras generalize both orthomodular algebras and Heyting algebras. In fact, we shall prove in the next section that the variety of Ex-algebras is the join variety of the variety of orthomodular algebras and the variety of Heyting algebras. We start by introducing a weaker notion.

\subsection{The axiom $\mathrm{(iEx)}$} \label{SectAxiomiEx}
\begin{definition}
A \textit{weak Ex-algebra} is a fundamental algebra in which the following inequality is valid:
\begin{align*}
\tag{$\mathrm{iEx}$}     \Bigg[\bigg(a &\wedge \Big((b\wedge c) \vee (b \wedge d)\Big)\bigg) \to g\Bigg] \wedge a \wedge (c \vee e) \wedge \lnot\lnot f \leq 
\lnot\lnot(a \wedge f) \\
    & \wedge \bigg[(a\wedge c) \vee (a \wedge e) \vee f\bigg] \wedge 
    \bigg[ \big( b \wedge (c \vee d)\big) \vee \Big(\big (  b \wedge (c \vee d) \big) \to g\Big)\bigg].
\end{align*}
\end{definition}

Axiom $\mathrm{(iEx)}$ is related to the axiom $\mathrm{(Ex)}$ introduced in \cite{AgMa26} and, as we shall see below, implies it. Let us first show that weak Ex-algebras generalize both orthomodular algebras and Heyting algebras.

\begin{lemma} \label{orthoiexlma}
\textup{(iEx)} is valid in orthomodular logic. Thus, every orthomodular algebra is a weak Ex-algebra.
\end{lemma}

\proof
We prove that the following strengthening of \textup{(iEx)} is valid in all orthomodular algebras:
\begin{align*}
     a \wedge \lnot\lnot f \leq
\lnot\lnot(a \wedge f) \wedge f \wedge 
    \bigg[ \big( b \wedge (c \vee d)\big) \vee \Big(\big (  b \wedge (c \vee d) \big) \to g\Big)\bigg].
\end{align*}
We need to show that $a \wedge \lnot\lnot f$ entails each of the three conjuncts in the consequent of the displayed axiom.
First, we have 
\[a \wedge \lnot\lnot f \leq a \wedge f \leq \lnot\lnot (a \wedge f)\]
by double-negation elimination. Similarly, we have
$\lnot\lnot f\leq f$.
Finally, the third conjunct is valid already, as, letting $p := (b \wedge (c \vee d))$, we have
\[1 = p \vee \lnot p \leq p \vee (\lnot p \vee (p \land g)) = p \vee (p \to g). \]
\endproof

\begin{lemma} \label{heytingiexlma}
\textup{(iEx)} is valid in intuitionistic logic. Thus, every Heyting algebra is a weak Ex-algebra.
\end{lemma}

\proof
We prove that the following strengthening of \textup{(iEx)} is valid in all Heyting algebras:
\begin{align*}
   \Bigg[\bigg(a &\wedge \Big((b\wedge c) \vee (b \wedge d)\Big)\bigg) \to (e \wedge g)\Bigg] \wedge a \wedge (c \vee e) \wedge \lnot\lnot f\\ &\leq 
\lnot\lnot(a \wedge f)
     \wedge \bigg[(a\wedge c) \vee (a \wedge e) \bigg] \wedge 
    \bigg[  \big (  b \wedge (c \vee d) \big) \to g\bigg].
\end{align*}
Intuitionistically, we have
\begin{align}\label{eqVInt1}
    a \wedge \lnot\lnot f \leq \lnot\lnot a\wedge \lnot\lnot f \leq \lnot\lnot (a \wedge f).
\end{align}
By distributivity, we have 
\begin{align}\label{eqVInt2}
a \wedge (c \vee e) \leq (a \wedge c) \vee (a \wedge e).
\end{align}

Lastly, we derive the following:
\begin{align}\label{eqVInt3}
\Bigg[\bigg(a \wedge \Big((b\wedge c) \vee (b \wedge d)\Big)\bigg) \to (e \wedge g)\Bigg] \wedge a \leq  \Big(b \wedge (c \vee d)\Big) \to g.
\end{align}

By residuation, \eqref{eqVInt3} follows from:
\begin{align} \label{eqVInt3'}\Bigg[\bigg(a \wedge \Big((b\wedge c) \vee (b \wedge d)\Big)\bigg) \to (e \wedge g)\Bigg] \wedge a \land  \Big(b \wedge (c \vee d)\Big) \leq g.\end{align}
Let $p= (b \land c) \lor (b \land d)$, by distributivity and \textit{modus ponens}, we have:
\begin{align*}
    \Big[(a \wedge p) \to (e \wedge g)\Big] \wedge a \land  \Big(b \wedge (c \vee d)\Big) &\leq \Big[(a \wedge p) \to (e \wedge g)\Big] \wedge a \wedge p \\ &\leq e \land g \\
    &\leq g,
\end{align*}
as desired.
Putting together \eqref{eqVInt1}  \eqref{eqVInt2}, and \eqref{eqVInt3}, we  derive the strengthening of \textup{(iEx)} stated at the beginning of the proof.
\endproof

We now focus on a particular class of weak Ex-algebras.

\begin{definition}
An \textit{Ex-algebra} is a fundamental algebra in which $\mathrm{(iEx)}$ and the following four axioms are valid:
\begin{align}
\lnot (a \to b) &= \lnot (\lnot a\vee (a\wedge b)) \tag{nS}\\
a \wedge (a\to b) &= a \wedge b \tag{MP}\\
(a\wedge b)\to b &= 1 \tag{M}\\
a \to (b \land c) &\leq a \to b \tag{W4}
\end{align}
\end{definition}

Just like weak Ex-algebras, Ex-algebras generalize both orthomodular algebras and Heyting algebras.

\begin{lemma} \label{orthoexallma}
Axioms $\mathrm{(nS)}$, $\mathrm{(MP)}$, $\mathrm{(M)}$ and $\mathrm{(W4)}$ are valid in all orthomodular algebras. Thus, every orthomodular algebra is an Ex-algebra.
\end{lemma}

\proof
In orthomodular lattices $\mathrm{(nS)}$ is equivalent to $a \to b = \lnot a \vee (a\wedge b)$, which holds by definition. 

We verify $\mathrm{(MP)}$. The inequality $a \wedge (a\to b) \leq b$ holds by the orthomodular law. Clearly we also have $a \wedge (a\to b) \leq a$, and thus $a \wedge (a\to b) \leq a \wedge b$. Conversely, we have  $a \wedge b \leq a$ and $a \wedge b \leq \lnot a \vee (a\wedge b) = a\to b$, and thus $a \wedge (a\to b) = a\wedge b$.
For $\mathrm{(M)}$, we have 
\begin{align*}
(a\wedge b) \to b
&= \lnot (a\wedge b) \vee ((a \wedge b) \wedge b)\\
&= \lnot (a\wedge b) \vee (a \wedge b) = 1,
\end{align*}
as desired.
Finally, we clearly have that 
\[\neg a \lor (a \land (b \land c)) \leq \neg a \lor (a \land b),\] from which $\mathrm{(W4)}$ follows.
\endproof

\begin{lemma} \label{heytingexallma}
Axioms $\mathrm{(nS)}$, $\mathrm{(MP)}$, $\mathrm{(M)}$ and $\mathrm{(W4)}$ are valid in all Heyting algebras. Thus, every Heyting algebra is an Ex-algebra.
\end{lemma}

\proof
We first prove $\mathrm{(MP)}$. Notice that $a \wedge (a\to b) \leq a$ and $a\wedge (a\to b)\leq b$. Conversely, $a \wedge b \leq a$ and $a\wedge b \leq a\to b$, where the latter inequality follows from the fact that $a \wedge (a\wedge b) \leq b$ and residuation.

For $\mathrm{(nS)}$, we first observe that the equalities 
\begin{align*}\neg a \lor (a \land b) &= (\neg a \lor a) \land (\neg a \lor b) \\ &= \neg a \lor b \\
&= a \to b
\end{align*} are classically valid, which means that $\neg (a\to b) = \neg (\neg a \lor (a \land b))$ is also classically valid. By Glivenko's theorem, this means that the inequalities \[\neg (a \to b) \leq \neg\neg\neg (\neg a \lor (a \land b)) = \neg(\neg a \lor (a \land b))\] and
\[\neg (\neg a \lor (a \land b)) \leq \neg\neg\neg(a\to b) = \neg (a\to b)\] are intuitionistically valid. This shows that $\mathrm{(nS)}$ is valid.
For $\mathrm{(M)}$, we have that $1 \land a \land b \leq b$, which by residuation implies $1 \leq (a \land b) \to b$, as desired.
Finally, for $\mathrm{(W4)}$, by residuation, it is enough to show that $a \land (a\to (b \land c)) \leq b$. But \textit{modus ponens} gives us $a \land (a \to (b \land c)) \leq b \land c \leq b$, as desired.
\endproof

\subsection{Properties of Ex-algebras and weak Ex-algebras}\label{SectProperties}

Let us now derive a number of properties of (weak) Ex-algebras, which will play an important role in the proof of our main theorem in the next section. We start by clarifying the relationship between the axioms $\mathrm{(Ex)}$ and $\mathrm{(iEx)}$.

\begin{lemma}
The $\{\wedge,\vee,\lnot\}$-reduct of every Ex-algebra is an Ex-lattice.
\end{lemma}
\proof
By setting $g = 0$ in $\mathrm{(iEx)}$, we recover $\mathrm{(Ex)}$, from which the claim follows.
\endproof

Next, we single out three additional axioms that follow directly from $\mathrm{(iEx)}$.

\begin{lemma}
    The following axioms are valid in all weak Ex-algebras:

    \begin{itemize}
        \item[{\textup{(iCl)}}]
$\Big[\big( a \land ((b \land c) \lor (b \land d))\big) \to g\Big] \land a \leq (b \land (c\lor d)) \lor \Big[(b\land (c \lor d)) \to g\Big]$;

\item[\textup{(iCl$_1$)}] $\big(( a \land b) \to g \big)\land a \leq b \lor (b \to g)$;

\item[\textup{(iCl$_2$)}] $\big( (b \land c) \lor (b \land d)\big) \to g  \leq (b \land (c\lor d)) \lor \big[(b\land (c \lor d)) \to g\big]$;

    \end{itemize}
\end{lemma}

\proof
We prove each item in order. Substituting $ e = f = 1$ in \textup{(iEx)}, we have
\[\Big[\big(a \land ((b \land c) \lor (b \land d))\big) \to g\Big] \land a \leq \neg \neg a \land \Big[\big((b \land (c\lor d)) \lor ((b\land (c \lor d)) \to g)\big)\Big],\]
which clearly implies \textup{(iCl)}.

Moreover, \textup{(iCl$_1$)} follows from substituting $c = d = 1$ in \textup{(iCl)}, and \textup{(iCl$_2$)} follows from substituting $a=1$ in \textup{(iCl)}.
\endproof

Finally, we prove two useful properties of Ex-algebras.

\begin{lemma}\label{LemmaInequalityEx}
Suppose $L$ is an Ex-algebra. Then, for all $a, b \in L$, we have
\[a \leq b \text{ if and only if } a \to b = 1.\]
\end{lemma}
\proof
Suppose that $a\leq b$. Then by $\mathrm{(M)}$ we have 
$1 = (a\wedge b) \to b = a \to b$.
Conversely, suppose that $a\to b = 1$. By $\mathrm{(MP)}$, we have 
$a\wedge b = a \wedge (a\to b) = a$,
which implies $a \leq b$.
\endproof

\begin{lemma}\label{LemmaWes2}
Suppose that $L$ is an Ex-algebra. Then, for all $a,b \in L$, we have
\[a\wedge b \leq a\to b.\]
\end{lemma}
\proof
By Lemma \ref{LemmaInequalityEx}, we have that $a\land b \leq a\land b$ implies $1 \leq (a \land b) \to (a\land b)$. Now we reason as follows:
\begin{align*}
1 &\leq (a\wedge b) \to (a\wedge b)\\
&\leq (a\wedge b) \to \big( (a \to b) \wedge a\big), &\text{by $\mathrm{(MP)}$}\\
&\leq (a\wedge b) \to (a\to b), &\text{ by $\mathrm{(W4)}$},
\end{align*}
from which the inequality follows by Lemma \ref{LemmaInequalityEx} again.
\endproof

We conclude this section with two examples. The $7$-element lattice depicted on the left below is an Ex-lattice that cannot be turned into an Ex-algebra. Red arrows represent the negation operation, with the convention that $\neg 0 = 1$ and $\neg 1 = 0$. The $5$-element lattice on the right, with the implication given by the adjacent table, is an Ex-algebra that is neither a Heyting algebra nor an orthomodular algebra. In particular, it has a ``trivial'' negation, where $\neg 0 =1$ and $\neg x = 0$ for any $x \neq 0$. 

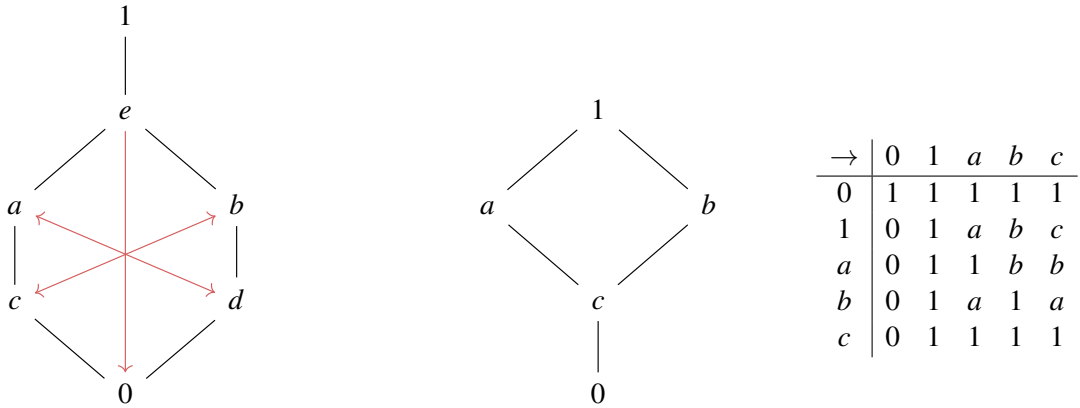
\begin{minipage}{0.4\textwidth}
    \[\begin{tikzcd}
	& 1 &&&&&& \\
	& e &&&&& 1 \\
	a && b &&& a && b \\
	c && d &&&& c \\
	& 0 &&&&& 0
	\arrow[no head, from=2-2, to=1-2]
	\arrow[color={rgb,255:red,214;green,92;blue,92}, from=2-2, to=5-2]
	\arrow[no head, from=3-1, to=2-2]
	\arrow[color={rgb,255:red,214;green,92;blue,92}, tail reversed, from=3-1, to=4-3]
	\arrow[no head, from=3-3, to=2-2]
	\arrow[no head, from=3-6, to=2-7]
	\arrow[no head, from=3-8, to=2-7]
	\arrow[no head, from=4-1, to=3-1]
	\arrow[color={rgb,255:red,214;green,92;blue,92}, tail reversed, from=4-1, to=3-3]
	\arrow[no head, from=4-3, to=3-3]
	\arrow[no head, from=4-7, to=3-6]
	\arrow[no head, from=4-7, to=3-8]
	\arrow[no head, from=5-2, to=4-1]
	\arrow[no head, from=5-2, to=4-3]
	\arrow[no head, from=5-7, to=4-7]
\end{tikzcd}\]
\end{minipage}
\hfill
\begin{minipage}{0.35\textwidth}
\vspace*{4em}

\[
\begin{array}{c|ccccc}
\to & 0 & 1 & a & b & c \\
\hline
0 & 1 & 1 & 1 & 1 & 1 \\
1 & 0 & 1 & a & b & c \\
a & 0 & 1 & 1 & b & b \\
b & 0 & 1 & a & 1 & a \\
c & 0 & 1 & 1 & 1 & 1
\end{array}
\]
\end{minipage}

\section{The iEx-Embedding Theorem}\label{Sectmainthm}
Throughout this section, a map between fundamental algebras that preserves all the operations will be called a \textit{fundamental algebra homomorphism}, and a map between the $\{\land,\lor,\neg\}$-reducts of two fundamental algebras that preserve all three operations will be called a \textit{fundamental homomorphism}. We are now ready to prove our main theorem. Recall that a homomorphism $f: L \to M_1 \times M_2$ is a subdirect embedding if $f$ is injective and the composition of $f$ with the projections $\pi_1: M_1 \times M_2 \to M_1$ and $\pi_2: M_1 \times M_2 \to M_2$ yields surjective homomorphisms.

\begin{theorem}[The iEx-Embedding Theorem] \label{mainthm}
Let $(L, \land, \lor, 0,1, \to)$ be an Ex-algebra. Then there are an orthomodular algebra $O_L$, a Heyting algebra $I_L$, and a subdirect fundamental algebra embedding $e: L \to O_L \times I_L$.
\end{theorem}

The theorem will be proved using the (lattice) Ex-Embedding Theorem from \cite{AgMa26}. There, we proved that if $(L, \wedge,\vee,0,1,\lnot)$ is an Ex-lattice (i.e., a fundamental lattice satisfying $\mathrm{(Ex)}$), then there are an ortholattice $O_L$, a Heyting lattice $I_L$, and a subdirect fundamental lattice embedding $e : L \to O_L \times I_L$.\\

We shall first recall this construction from \cite{AgMa26}.
Let $L$ be an Ex-lattice. We define an ortholattice $O_L$ and a Heyting lattice (i.e., a weakly pseudo-complemented distributive lattice) $I_L$ which are both homomorphic images of $L$. An element $a \in L$ is then mapped to a pair $e(a) = (a_O, a_I) \in O_L\times I_L$ which we describe below. 
The ortholattice $O_L$ is defined as a quotient of $L$ by the equivalence relation
\[a \sim b \leftrightarrow \lnot a = \lnot b.\]
We shall write $[a]_O$ for the equivalence class of $a$ (i.e., for $a_O$) if this helps with clarity.
It was shown in \cite{AgMa26} that $O_L$ is indeed an ortholattice and that this equivalence relation is congruent with the operations $\wedge, \vee$, and $\lnot$, i.e., that $[a \wedge b]_O = a_O \wedge b_O, [a\vee b]_O = a_O\vee b_O$, and $\lnot a_O = [\lnot a]_O$.

The Heyting lattice $I_L$ is defined as follows. Letting $(P(L), \rset)$ be the set of (proper) prime filters of $L$ ordered by reverse inclusion, we consider the $\to$-free reduct of the Heyting algebra $Dn(P(L))$ of open subsets of $(P(L), \rset)$ endowed with the downset topology. We then define $I_L$ to be the Heyting sublattice of $Dn(P(L))$ generated by sets of the form $a_I = \{P \in P(L) \mid a \in P\}$ for any $a \in L$, and show that the map $a \mapsto a_I$ is a surjective fundamental homomorphism. \\

We now extend these results to the setting of Ex-algebras. For the rest of this section, fix an Ex-algebra $(L, \land, \lor, 0, 1, \to)$. We first define $O_L$ as the quotient of $L$ determined by the equivalence relation $a \sim b \Leftrightarrow \neg a = \neg b$ and write $a_O$ for the equivalence class of an element $a$ modulo $\sim$.
We first claim that $\sim$ is a congruence relation on $L$. Because $(L, \land, \lor, 0,1, \neg)$, with $\neg a$ defined as $a \to 0$, is an Ex-lattice, we know already that $\sim$ preserves the operation $\land$, $\lor$ and $\neg$. We now claim the following:

\begin{lemma}
For all $a, b,c,d\in L$, if $a_O = c_O$ and $b_O = d_O$, then $[a\to b]_O = [c\to d]_O$.
\end{lemma}

\proof
Suppose that $a_O = c_O$ and $b_O = d_O$. We must show that $\neg (a \to b) = \neg(c\to d)$. Because $L$ is an Ex-algebra, it satisfies $\mathrm{(nS)}$. Hence we have $\neg (a \to b) = \neg (\neg a \lor (a \land b))$ and $\neg (c \to d) = \neg( \neg c \lor (c \land d))$.
Since $\sim$ preserves $\land$, $\lor$ and $\lnot$, we have that $a \sim c$ and $b \sim d$ entails
\begin{align*}
    a \to b &\sim \neg a \lor (a \land b) \sim \neg c \lor (c \land d)\sim c \to d,
\end{align*}
as desired. This shows that $[a \to b]_O = [c \to d]_O$.
\endproof

\begin{lemma} \label{orthomorphlma}
$(O_L, \wedge,\vee,0,1,\to)$ is a homomorphic image of $(L, \land,\lor,0,1,\to)$. Moreover, it is an orthomodular algebra.
\end{lemma}
\proof
For the first part, we know already that the map $a \mapsto a_O$ is a surjective lattice homomorphism that preserves $\neg$. Hence we only need to show that $[a \to b]_O = a_O \to b_O$ for any $a, b\in L$. Since \[a_O \to b_O = \neg a_O \lor (a_O \land b_O) = [\neg a \lor (a \land b)]_O,\] this amounts to showing that $a\to b \sim \neg a \lor (a \land b),$ which holds by $\mathrm{(nS)}$.
For the second part, we know already that $O_L$ is an ortholattice. By the result of Dalla Chiara and Giuntini \cite{DCGi02}, it then suffices to show that the principle $a \wedge (\lnot a \vee (a\wedge b)) \leq b$ is valid in $O_L$. But this is immediate from the fact that $\mathrm{(MP)}$ holds in $L$ together with the first part. Hence $O_L$ is an orthomodular algebra.
\endproof

We have shown that the mapping $a \mapsto a_O$ is a (surjective) fundamental algebra homomorphism. We now turn to $I_L$. Following \cite{AgMa26}, we consider first the Heyting algebra $Dn(P(L))$ of open subsets of the topological space $(P(L), \rset)$, where $P(L)$ is the set of all (proper) prime filters on $L$ and open sets are subsets of $P(L)$ which are downward-closed (here, a prime filter $Q$ is below $P$ if $P \sset Q$). We let $a \mapsto a_I$ maps any $a \in L$ to the set $a_I = \{P \in P(L) \mid a \in P\}$, and define $I_L$ as the Heyting subalgebra of $Dn(P(L))$ generated by sets of the form $a_I$ for some $a \in A$. We now claim that the map $a \mapsto a_I$ is a surjective fundamental algebra homomorphism. It is routine to check that this map preserves $\land$ and $\lor$, but we must show that $[a \to b]_I = a_I \to b_I$, where $a_I \to b_I$ is the Heyting implication in $Dn(P(L)$, i.e.: \[a_I \to b_I = \{P \in P(L) \mid \forall Q( (Q \rset P \land a \in Q)\to b \in Q)\}.\]

We start by proving the following restricted version of the Prime Filter Theorem.

\begin{lemma} \label{restPFT}
    Let $L$ be an Ex-algebra, $P$ a prime filter on $L$ and $F$ and $I$ a filter and an ideal on $L$ respectively. If $P \sset F$ and $F \cap I = \emptyset$, then there is a prime filter $Q$ on $L$ such that $F \sset Q$ and $Q \cap I = \emptyset$. 
\end{lemma}

\proof 
Let $P$, $F$ and $I$ as in the statement of the lemma, and let $Q$ be a maximal filter extending $F$ and disjoint from $I$, which exists by Zorn's Lemma. We claim that $Q$ is prime. To see this, suppose that $c \notin Q$ and $d \notin Q$. We want to show that $c \lor d \notin Q$. Since $Q$ is maximal among filters disjoint from $I$, there must be $a \in Q$ and $b \in I$ such that $a \land c \leq b$ and $a \land d \leq b$. This means that $(a \land c) \lor (a \land d) \leq b$ and thus, by, Lemma \ref{LemmaInequalityEx} together with $\mathrm{iCl2}$, we have \[1 \leq ((a \land c) \lor (a \land d)) \to b \leq (a \land (c\lor d)) \lor ((a \land (c\lor d)) \to b).\] Since $1 \in P$ and $P$ is prime, this means that either $a \land (c \lor d) \in P$ or $(a \land (c \lor d)) \to b \in P$. In the first case, we have that $c \lor d \in P$, and hence either $c \in P \sset Q$ or $d \in P \sset Q$, contradicting our assumption. In the second case, we have that $(a\land (c\lor d)) \to b \in P \sset Q$. But then, since $a \in Q$, $\mathrm{(MP)}$ entails that $c\lor d \in Q$ implies $b \in Q$. Hence $c \lor d \notin Q$, and $Q$ is prime, as desired.
\endproof

The following is a straightforward application of the previous lemma.

\begin{lemma}\label{LemmaExEmbeddingiCl}
Let $L$ be an Ex-algebra. Suppose that $P$ is a prime filter on $L$ such that $a\to b \not \in P$. Then, there is a prime filter $Q$ extending $P$ such that $a \in Q$ but $b \not \in Q$.
\end{lemma}
\proof
Suppose $P$ is a prime filter not containing $a \to b$. If $a \vee (a\to b) \in P$, then $a \in P$, as $P$ is prime. We must already have $b \not \in P$, as otherwise we would have $a\to b \in P$, since $a\wedge b \leq a\to b$ by Lemma \ref{LemmaWes2}.
Thus, we may assume that $a \vee (a\to b)\not\in P$. We claim that $P \cup \{a\}$ generates a filter which does not contain $b$. If not, then there is $c \in P$ such that $a \wedge c \leq b$, so $(a\wedge c)\to b = 1$, by Lemma \ref{LemmaInequalityEx}. Thus, $((a \wedge c) \to b) \wedge c \in P$.
By $\mathrm{(iCl_1)}$, we have $((a \wedge c) \to b) \wedge c\leq a \vee (a\to b)$,
so $a \vee (a\to b) \in P$, contrary to the assumption above.

Thus, $P \cup\{a\}$ generates a filter disjoint from the ideal generated by $\{b\}$. By Lemma \ref{restPFT}, there is a prime filter $Q$ extending $P \cup \{a\}$ that does not contain $b$, as desired.
\endproof

\begin{lemma} \label{heytmorphlma}
The mapping $a \to a_I$ is a surjective fundamental algebra homomorphism. Moreover, if $\neg a = \neg b$ and $a \nleq b$, then $a_I \not \sset b_I$.
\end{lemma}
\proof
For the first part, we only need to show that if $P \in P(L)$, $a\to b \in P$ if and only if for all $Q$ extending $P$ with $a \in Q$, we have $b \in Q$. The right-to-left direction is Lemma \ref{LemmaExEmbeddingiCl}. The left-to-right direction follows because any such $Q$ is a filter on $L$ containing $a\to b$ and $a$, thus also $a \wedge (a \to b)$. Since $a \wedge (a\to b) \leq b$, $Q$ also contains $b$. For the second part, Lemma 4.7 in \cite{AgMa26} establishes that whenever $\neg a = \neg b$ and $a \nleq b$ in an Ex-lattice $L$, there is a prime filter $P$ over $L$ containing $a$ and not containing $b$. Since Ex-algebras are also Ex-lattices, the result applies here as well.
\endproof 

Our main theorem now follows: since each mapping $a \mapsto a_O$ and $a \mapsto a_I$ is a fundamental algebra homomorphism, the map $e: L \to O_L \times I_L$ given by $e(a) = (a_O, a_I)$ is also a fundamental algebra homomorphism. To see that it is injective, suppose that $a \nleq b$. We distinguish two cases. If $\neg a \neq \neg b$ then $a_O \nleq b_O$. Otherwise, by Lemma \ref{heytmorphlma}, we have $a_I \not \sset b_I$. Either way, it follows that $e(a) \nleq e(b)$, as desired. Lastly, since the composition of $e$ with the projections $\pi_1: O_L\times I_L \to O_L$ and $\pi_2: O_L \times I_L \to I_L$ coincides with the surjective maps $a \mapsto a_O$ and $a \mapsto a_I$ respectively, it follows that $e$ is a subdirect embedding. This completes the proof of the iEx-Embedding Theorem.
\endproof

\section{Implicative Constructive Quantum Logics}\label{icqsec}
In this section, we derive some consequences of the iEx-Embedding Theorem established above for constructive quantum logics equipped with an implication. We shall call such logics \textit{implicative constructive quantum logics} ($ICQ$ logics for short).

\subsection{The Minimal $ICQ$ Logic}
We start by axiomatizing the intersection of intuitionistic logic and orthomodular logic, which can arguably be considered the minimal $ICQ$ logic.

\begin{theorem}\label{TheoremEquivExalgebras}
Let $\varphi$ and $\psi$ be formulas.
The following are equivalent:
\begin{enumerate}
\item $\varphi\vdash\psi$ is valid in all Heyting algebras and in all orthomodular algebras.
\item $\varphi\vdash\psi$ is valid in all Ex-algebras.
\end{enumerate}
\end{theorem}
\proof
It was verified in Section \ref{SectAxiomiEx} that every Heyting algebra and every orthomodular algebra is an Ex-algebra. Conversely, suppose that $\varphi\vdash \psi$ is not valid in all Ex-algebras, so that $v(\varphi) \nleq v(\psi)$ holds for some Ex-algebra $L$ and some valuation $v$. By the iEx-Embedding Theorem, there is an orthomodular algebra $O_L$ a Heyting Algebra $I_L$ and a subdirect embedding
$e: L \to O_L\times I_L$. By a standard argument in algebraic logic, this implies that the entailment $\varphi\vdash\psi$ is falsifiable in either $O_L$ or $I_L$, as desired.
\endproof

As a consequence of Theorem \ref{TheoremEquivExalgebras}, we immediately obtain:
\begin{theorem}
The joint validities of intuitionistic logic and orthomodular logic are axiomatized over Fundamental Logic with implication by $\mathrm{(iEx)}$ together with the following principles:
\begin{enumerate}
\item[\textup{(nS)}] $\lnot (a \to b) = \lnot(\lnot a \vee (a\wedge b))$;
\item[\textup{(MP)}] $a \wedge (a \to b) = a \wedge b$;
\item[\textup{(M)}] $(a \wedge b) \to b = 1$;
\item[\textup{(W4)}] $a \to (b \land c) \leq a \to b$.
\end{enumerate}
\end{theorem}
\proof
This follows from Theorem \ref{TheoremEquivExalgebras} together with the definition of Ex-algebras, the completeness of intuitionistic logic with respect to Heyting algebras, and the completeness of orthomodular logic with respect to orthomodular algebras (see Dalla Chiara and Giuntini \cite{DCGi02}).
\endproof 

Analyzing the proof of the iEx-Embedding Theorem, we also obtain the following:
\begin{theorem} \label{altaxthm}
The joint validities of intuitionistic logic and orthomodular logic with implication are axiomatized over Fundamental Logic with implication by the following principles:
\begin{enumerate}
\item[\textup{(nS)}] $\lnot (a \to b) = \lnot (\lnot a\vee (a\wedge b))$.
\item[\textup{(MP)}] $a \wedge (a\to b) = a \wedge b$.
\item[\textup{(M)}] $(a\wedge b)\to b = 1$.
\item[\textup{(W4)}] $a \to (b\wedge c)\leq a \to b$.
\item[\textup{(iCl$_1$)}] $a\wedge ((a\wedge b) \to c) \leq b \vee (b\to c)$.
\item[\textup{(iCl$_2$)}] $((b \wedge c) \vee (b\wedge d)) \to a \leq (b \wedge (c\vee d)) \vee \big((b \wedge (c\vee d)) \to a\big)$.
\item[\textup{(Nu)}] $\lnot\lnot a \wedge \lnot\lnot b \leq \lnot\lnot (a\wedge b)$.
\item[\textup{(Vi)}] $a \wedge (b\vee c)\wedge \lnot\lnot d \leq (a \wedge b) \vee (a\wedge c) \vee d$.
\end{enumerate}
\end{theorem}
\proof
It was verified throughout Section \ref{SectProperties} that the principles \textup{(iCl$_1$)}, \textup{(iCl$_2$)}, \textup{(Nu)}, and \textup{(Vi)} can all be obtained as instances of \textup{(iEx)}. 

Conversely, examining the proof of the iEx-Embedding Theorem, we track down all uses of \textup{(iEx)}: We see that the proof in Section \ref{Sectmainthm} only makes use of the principles \textup{(MP)}, \textup{(nS)}, \textup{(iCl$_1$)}, and \textup{(iCl$_2$)}. The principles \textup{(Nu)} and \textup{(Vi)} are used implicitly when appealing to the lattice Ex-Embedding Theorem of \cite{AgMa26}. Finally, the proof makes use of Lemma \ref{LemmaInequalityEx}, which only requires \textup{(M)} and \textup{(MP)}, and Lemma \ref{LemmaWes2}, which again only makes use of \textup{(M)} and \textup{(MP)}, as well as \textup{(W4)}.
\endproof

\subsection{The Lattice of $ICQ$ Logics} \label{laticqsec}

Next, we generalize our previous result by providing an axiomatization for the intersection $\mathsf{L}_O \cap \mathsf{L}_I$ of any pair of a superintuitionistic logic $\mathsf{L}_I$ and a superorthomodular logic $\mathsf{L}_O$. As a corollary, we also obtain a characterization of the lattice of all $ICQ$ logics, i.e., the lattice of all extensions of iEx-logic. A similar result was already established for extensions of Ex-logic, the lattice of which is isomorphic to the product of the lattice of quantum logics and the lattice of extensions of the implication-free fragment of intuitionistic logic \cite[Thm.~6.7]{AgMa26}. The proof of the result presented here follows the same strategy. We start with the following two lemmas. 

\begin{lemma} \label{isolma1}
Let $\mathsf{L}_I$ be a superintuitionistic logic axiomatized by $\{\phi_i \vdash \psi_i\}_{i \in I}$ and $\mathsf{L}_O$ be a superorthomodular logic axiomatized by $\{\phi_j \vdash \psi_j\}_{j \in J}$. Then $\mathsf{L}_I \cap \mathsf{L}_O$ is axiomatized by extending iEx-logic with the set of axioms \[\{\phi_i \vdash \psi_i \lor \neg \psi_i\}_{i \in I} \cup \{\phi_j \vdash \neg \neg \psi_j\}_{j \in J}.\]
\end{lemma}

\proof
Fix a superintuitionistic logic $\mathsf{L}_I$ axiomatized by $\{\phi_i \vdash \psi_i\}_{i \in I}$ and a superorthomodular logic $\mathsf{L}_O$ axiomatized by $\{\phi_j \vdash \psi_j\}_{j \in J}$. We first check that the set $\{\phi_i \vdash \psi_i \lor \neg \psi_i\}_{i \in I} \cup \{\phi_j \vdash \neg \neg \psi_j\}_{j \in J}$ is a set of valid axioms in $\mathsf{L}_I \cap \mathsf{L}_O$. Fix some $i \in I$. Clearly, $\mathsf{L}_I$ contains the axiom $\phi_i \vdash \psi_i \lor \neg \psi_i$, since $\phi_i \vdash \psi_i \in \mathsf{L}_I$. Moreover, $\phi_i \vdash \psi_i \lor \neg \psi_i$ follows from the excluded middle, so it also belongs to $\mathsf{L}_O$. Now consider an axiom of the form $\phi_j \vdash \neg \neg \psi_j$ for some $j \in J$. Over orthomodular logic, this axiom is equivalent to $\phi_j \vdash \psi_j$, hence it also belongs to $\mathsf{L}_O$. Moreover, since $\mathsf{L}_O$ is a subclassical logic, $\phi_j \vdash \psi_j$ is classically valid. By Glivenko's theorem, this means that $\phi_j \vdash \neg \neg \psi_j$ is valid intuitionistically, and therefore belongs to $\mathsf{L}_I$, as desired. This shows that $\{\phi_i \vdash \psi_i \lor \neg \psi_i\}_{i \in I} \cup \{\phi_j \vdash \neg \neg \psi_j\}_{j \in J}$ is a set of axioms in $\mathsf{L}_I \cap \mathsf{L}_O$. 

Let now $\mathsf{L}$ be the extension of \textup{iEx} axiomatized by this set of axioms. Clearly, we have that $\mathsf{L} \sset \mathsf{L}_I \cap \mathsf{L}_O$. We now claim that $\mathsf{L}_I \cap \mathsf{L}_O \sset \mathsf{L}$. To show this, suppose that $\phi \vdash \psi$ is not valid in $\mathsf{L}$. Then there is an $\mathsf{L}$-lattice $L$ such that $\phi \vdash \psi$ is not valid in $L$. By the iEx-Embedding Theorem, $L$ subdirectly embeds into $O_L \times I_L$. Since $O_L$ is a homomorphic image of $L$, the axiom $\phi_j \vdash \neg \neg\psi_j$ is valid on $O_L$ for any $j\in J$. But since $O_L$ is an ortholattice, this means that $\phi_j \vdash \psi_j$ is also valid on $O_L$ for any $j \in J$, and therefore $O_L$ is an $\mathsf{L}_O$-lattice. Similarly, we have that $I_L$ is a homomorphic image of $L$, and therefore $\phi_i \vdash \psi_i \lor \neg \psi_i$ is valid in $I_L$ for any $i \in I$. Fix some $i \in I$. Since $\phi_i \vdash \psi_i$ is classically valid, we have that $\vdash \neg (\phi_i \land \neg\psi_i)$ is also classically valid, and thus also intuitionistically valid by Glivenko's theorem. This means that, for any valuation $V$ on $I_L$, we have $V(\phi_i) \leq V(\psi_i) \lor V(\neg\psi_i)$ and $(V(\phi_i) \land \neg V(\psi_i)) \leq 0$. By distributivity, this gives us:
\[V(\phi_i) \leq V(\phi_i) \land V(\psi_i \lor \neg \psi_i) \leq (V(\phi_i) \land V(\psi_i)) \lor (V(\phi_i)\land \neg V(\psi_i)) \leq V(\phi_i) \land V(\psi_i),\] from which it follows that $\phi_i \vdash \psi_i$ is valid on $I_L$ for any $i \in I$. Hence $I_L$ is an $\mathsf{L}_I$-lattice. Since $\phi \vdash \psi$ is not valid on $L$ and embeds into $O_L \times I_L$, $\phi \vdash \psi$ is not valid on $O_L \times I_L$, which means that it is not valid on either $O_L$ or $I_L$. Either way, it follows that $\phi \vdash \psi \notin \mathsf{L}_O \cap \mathsf{L}_I$. This shows that $\mathsf{L}_O \cap \mathsf{L}_I \sset \mathsf{L}$, as desired.
\endproof

The previous lemma shows how to axiomatize the intersection of any pair of a superorthomodular logic and a superintuitionistic logic as an extension of iEx-logic. In fact, every extension of iEx-logic arises as such an intersection of logics, as the following establishes.

\begin{lemma} \label{isolma2}
    Let $\mathsf{L}$ be an extension of $\textup{iEx}$, and let $\mathbb{V}(\mathsf{L})$ be the variety of fundamental algebras on which $\mathsf{L}$ is valid. Let $O_\mathsf{L}$ be the logic of the class of algebras $\{O_L \mid L \in \mathbb{V}(\mathsf{L})\}$ and $I_\mathsf{L}$ be the logic of the class of algebras $\{I_L \mid L \in \mathbb{V}(\mathsf{L})\}$, where, for any Ex-algebra $L$, $O_L$ and $I_L$ are as in the iEx-Embedding Theorem. Then $\mathsf{L} = O_\mathsf{L} \cap I_\mathsf{L}$.
\end{lemma}

\proof
Clearly, since $O_L \models \mathsf{L}$ and $I_L \models \mathsf{L}$ whenever $L \models \mathsf{L}$, we have that $\mathsf{L} \sset O_\mathsf{L} \cap I_\mathsf{L}$. For the converse, note that, if $\phi \vdash \psi \in O_\mathsf{L} \cap I_\mathsf{L}$, then $\phi \vdash \psi$ is valid on $O_L \times I_L$ for any $L \in \mathbb{V}(\mathsf{L})$. By the iEx-Embedding theorem, $L$ subdirectly embeds into $O_L \times I_L$, so $\phi \vdash \psi$ is valid on $L$. This shows that $O_\mathsf{L} \cap I_\mathsf{L} \sset \mathsf{L}$.
\endproof

As a straightforward consequence of the previous two lemmas, we can provide a characterization of the lattice of extensions of iEx-logic.

\begin{theorem}
    Let $i\mathbf{SI}$ be the lattice of superintuitionistic logics, $i\mathbf{SO}$ the lattice of superorthomodular logics, and $i\mathbf{SE}$ the lattice of $ICQ$ logics. Then $i\mathbf{SE} \simeq i\mathbf{SO} \times i\mathbf{SI}$.
\end{theorem}

\proof
Consider the map $(\mathsf{L}_O,\mathsf{L}_I) \mapsto \mathsf{L}_O \cap \mathsf{L}_I$ from the lattice $i\mathbf{SO} \times i\mathbf{SI}$ to the lattice $i\mathbf{SE}$. Clearly, this map is monotone, and it is surjective by Lemma \ref{isolma2}. Hence, it is an isomorphism if we can also show injectivity. Suppose that $\mathsf{L}_I \sset \mathsf{L}_I' \in i\mathbf{SI}$ and $\mathsf{L}_O \sset \mathsf{L}_O' \in i\mathbf{SO}$ are such that $(\mathsf{L}_O,\mathsf{L}_I) \nleq_{i\mathbf{SO}\times i\mathbf{SI}} (\mathsf{L}_O',\mathsf{L}_I')$. Then either $\mathsf{L}_O' \not \sset \mathsf{L}_O$ or $\mathsf{L}_I' \not \sset \mathsf{L}_I$. In the first case, this means that there is an $\mathsf{L}_O$-orthomodular algebra $O$ and an axiom $\phi \vdash \psi \in \mathsf{L}_O'$ that is not valid on $O$. By Lemma \ref{isolma1}, we have that $\phi \vdash \neg \neg \psi$ is not valid on $O$, yet $\phi \vdash \neg \neg \psi \in \mathsf{L}_O' \cap \mathsf{L}_I'$. In the second case, there is a $\mathsf{L}_I$-Heyting algebra $A$ and an axiom $\phi \vdash \psi$ that is not valid on $A$. By \ref{isolma1} again, we have that $\phi \vdash \psi \lor \neg \psi$ is not valid on $A$, yet $\phi \vdash \psi \lor \neg \psi \in \mathsf{L}_O' \cap \mathsf{L}_I'$. Either way, this shows that $\mathsf{L}_O' \cap \mathsf{L}_I' \not \sset \mathsf{L}_O \cap \mathsf{L}_I$, as desired.
\endproof
\section{Adding Modalities} \label{modalsec}
In this final section, we consider extensions of fundamental algebras with modal operators. In particular, we are interested in using the iEx-Embedding Theorem to axiomatize the intersection of a reasonable intuitionistic modal logic and a reasonable quantum modal logic. Intuitionistic modal logic is a sprawling and fast growing field \cite{Si94,SteDePa08}. Prominent candidates to the title of ``minimal intuitionistic modal logic'' include Wijesekera's logic $\mathsf{CK}$ \cite{Wij90} and Fischer Servi's logic $\mathsf{IK}$ \cite{FS84}. The field of quantum modal logics is comparatively less explored \cite{Do09,HoMa24,To25}. Here, we do not intend to make any claim regarding minimal intuitionistic modal logic or minimal quantum modal logic, but we are rather interested in logics whose intersection can be axiomatized using the techniques presented here. We leave a more thorough investigation of constructive quantum modal logics, e.g. along the lines of Holliday's fundamental modal logic \cite{Ho24}, for future work.

From the point of view of the $\textup{iEx}$-Embedding Theorem, the addition of modalities to the signature of fundamental algebras introduces some significant difficulties. In particular, the representation of a fundamental modal algebra as a subalgebra of the product of a quantum modal algebra and an intuitionistic modal algebra obtained as quotients of the original algebra is only possible if the quotient maps preserve the modal operators. Here, we run into some significant issues. In algebras corresponding to Fischer Servi's logic $\mathsf{IK}$ (and, \textit{a fortiori}, in algebras corresponding to the weaker logic $\mathsf{CK}$), it is not always the case that the equivalence relation $\sim$ defined above preserves $\Box$. In other words, $\neg a= \neg b$ does not always imply that $\neg \Box a = \neg \Box b$. Similarly, the proof that the map $a \mapsto a_I$ also preserves modal operators requires a stronger version of the restricted Prime Filter Theorem \ref{restPFT} that may not hold in modal orthomodular lattices. Here, we overcome these issues by restricting ourselves to classes of algebras in which they do not arise, leaving the development of a more elegant solution for future work.

We start with the following definitions. 
\begin{definition}
    A $\Box$-involutive Fischer Servi algebra ($IFS$ algebra for short) is a tuple $(A,\land,\lor,0,1,\to,\Box,\Diamond)$ such that $(A,\land,\lor,0,1,\to)$ is a Heyting algebra and $\Box, \Diamond$ are unary operators satisfying the following:
    \begin{itemize}
        \item[$\textup{(M1)}$] $\Box 1 = 1$, $\Box(a \land b) = \Box a \land \Box b$;
        \item[$\textup{(M2)}$] $\Diamond 0 = 0$, $\Diamond (a \lor b) = \Diamond a \lor \Diamond b$;
        \item[$\textup{(FS1)}$] $\Diamond a \to \Box b \leq \Box (a \to b)$;
        \item[$\textup{(FS2)}$] $\Diamond (a \to b) \leq \Box a \to \Diamond b$;
        \item[$\textup{(I1)}$] $1 \leq \Box(\neg \neg a \to a)$.
    \end{itemize}
\end{definition}

The class of modal Heyting algebras satisfying the first four axioms corresponds to Fischer Servi's logic $\textup{IK}$. As mentioned above, axiom $\textup{(I1)}$ is need to guarantee that $\neg a = \neg b$ implies $\neg \Box a = \neg \Box b$, as evidenced by the Fischer Servi algebra depicted below, where $\neg 1 = \neg b$, but $\Box \neg 1 \neq \Box \neg b$. Blue and red arrows respectively represent the operators $\Diamond$ and $\Box$:
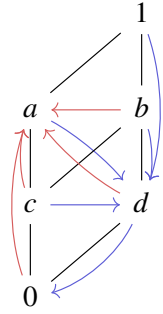
\[\begin{tikzcd}[scale=0.7]
	& 1 \\
	a & b \\
	c & d \\
	0
	\arrow[color={rgb,255:red,92;green,92;blue,214}, curve={height=-12pt}, from=1-2, to=3-2]
	\arrow[no head, from=2-1, to=1-2]
	\arrow[no head, from=2-1, to=3-1]
	\arrow[color={rgb,255:red,92;green,92;blue,214}, curve={height=-6pt}, from=2-1, to=3-2]
	\arrow[no head, from=2-2, to=1-2]
	\arrow[color={rgb,255:red,214;green,92;blue,92}, from=2-2, to=2-1]
	\arrow[no head, from=2-2, to=3-1]
	\arrow[color={rgb,255:red,92;green,92;blue,214}, curve={height=-6pt}, from=2-2, to=3-2]
	\arrow[color={rgb,255:red,214;green,92;blue,92}, curve={height=-6pt}, from=3-1, to=2-1]
	\arrow[color={rgb,255:red,92;green,92;blue,214}, from=3-1, to=3-2]
	\arrow[no head, from=3-1, to=4-1]
	\arrow[color={rgb,255:red,214;green,92;blue,92}, curve={height=-6pt}, from=3-2, to=2-1]
	\arrow[no head, from=3-2, to=2-2]
	\arrow[no head, from=3-2, to=4-1]
	\arrow[color={rgb,255:red,92;green,92;blue,214}, curve={height=-12pt}, from=3-2, to=4-1]
	\arrow[color={rgb,255:red,214;green,92;blue,92}, curve={height=-12pt}, from=4-1, to=2-1]
\end{tikzcd}\]

\begin{definition}
    A $\Diamond$-distributive orthomodular algebra ($\mathrm{DOM}$ algebra for short) is a tuple $(A,\land,\lor,0,1,\to,\Box,\Diamond)$ such that $(A,\land,\lor,0,1,\to)$ is an orthomodular algebra and $\Box, \Diamond$ are unary operators satisfying axioms $\textup{(M1)}$ and $\textup{(M2)}$ as well as the following:
    \begin{itemize}
        \item[$\textup{(M3)}$] $\Diamond \neg a = \neg \Box a$;
        \item[$\textup{(D1)}$] $\Diamond a \land \Box b \leq \Diamond (a \land b)$;
        \item[$\textup{(D2)}$] $\Diamond (a \land (b\lor c)) \leq \Diamond (a \land b) \lor \Diamond (a \land c)$.
    \end{itemize}
\end{definition}

Here again, the definition is more restrictive than in standard modal orthologic (see, e.g., \cite{HoMa24}). Axiom $\textup{(D1)}$ is a principle of positive modal logic, and $\textup{(D2)}$ amounts to assuming a form of ``modal distributivity'' which we will need later on. We now define the class of fundamental modal algebras that we will focus on from now on.

\begin{definition}
    A modal $\textup{iEx}$ algebra ($\textup{miEx}$ algebra for short) is a tuple $(A,\land,\lor,0,1,\to,\Box,\Diamond)$ such that $(A,\land,\lor,0,1,\to)$ is an $\textup{iEx}$ lattice and $\Box,\Diamond$ are unary operators satisfying axioms $\textup{(M1)}$, $\textup{(M2)}$, $\textup{(l1)}$, $\textup{(D1)}$, $\textup{(D2)}$ and the following:
    \begin{itemize}
        \item[$\textup{(MCl)}$] $\Box(((a \land b) \lor (a \land c)) \to d) \leq \Box((a \land (b \lor c)) \to d) \lor \Diamond(a \land b) \lor \Diamond(a \land c)$;
        \item[$\textup{(M4)}$] $\Diamond \neg a \leq \neg \Box a$;
        \item[$\textup{(M5)}$] $\neg\Diamond a = \Box \neg a$;
        \item[$\textup{(M6)}$] $\neg \Diamond \neg a \leq \Box a$;
        \item[$\textup{(FS3)}$] $\Diamond (a \to b) \leq (\Box a \to \Diamond b) \lor \Box a$.
    \end{itemize}
\end{definition}

We first verify that $miEx$ algebras generalize both $\mathrm{DOM}$ algebras and $\mathrm{IFS}$ algebras. 

\begin{lemma} \label{domlma}
    Any $\mathrm{DOM}$ algebra is a $\textup{miEx}$ algebra.
\end{lemma}

\proof
Note first that $\textup{(l1)}$ follows from $\textup{(M1)}$ together with $1 \leq a \to a =\neg\neg a \to a$, which holds because $\neg$ is an involution. Moreover, $\textup{(M4)}$, $\textup{(M5)}$ and $\textup{(M6)}$ all clearly follow from $\textup{(M3)}$ in any ortholattice. For $\textup{(FS3)}$, we have that $\neg \Box a \leq \Box a \to \Diamond b$, so the right-hand side also follows from the excluded middle. Hence it only remains to be shown that $\textup{(MCl)}$ holds. We claim that \[\neg \Diamond (a \land (b\lor c)) \leq \Box((a \land (b \lor c)) \to d).\] By $\textup{(D2)}$, it is enough, since then we have
\[1 \leq \neg\Diamond(a \land (b \lor c)) \lor \Diamond (a \land b \lor c) \leq \Box((a \land (b \lor c)) \to d) \lor \Diamond(a \land b) \lor \Diamond(a \land c).\]
For the proof of the claim, letting $p = (a \land (b \lor c))$, note that \[\neg\Diamond p \leq \Box\neg p \leq \Box(\neg p \lor (p \land d)) \leq \Box (p \to d),\] as desired.
\endproof

\begin{lemma} \label{ifslma}
    Any $\mathrm{IFS}$ algebra is a $\textup{miEx}$ algebra.
\end{lemma}

\proof
For $\textup{(D1)}$, note first that $\textup{(M2)}$ and $a \leq b \to (a\land b)$ yield $\Diamond a \to \Diamond(b \to (a \land b))$. By $\textup{(FS2)}$, this means that \[\Diamond a \leq \Diamond(b \to (a \land b)) \leq \Box b \to \Diamond(a \land b),\] and $\textup{(D1)}$ now follows from residuation. $\textup{(D2)}$ clearly follows from distributivity and $\textup{(D2)}$.
For $\textup{(MCl)}$, this follows from \[\Box(((a \land b) \lor (a \land c)) \to d) \leq \Box((a \land (b \lor c)) \to d),\] which itself holds in any $\textup{IFS}$ algebra by distributivity. For $\textup{(M4)}$, by $\textup{(FS2)}$, we have \[\Diamond (a \to 0) \land \Box a \leq \Box a \land (\Box a \to \Diamond 0) \leq \Diamond 0 \leq 0,\] whence $\Diamond \neg a \leq \neg \Box a$. For $\textup{(M5)}$, notice first that $\textup{(M4)}$ entails that $\Box \neg a \leq \neg \Diamond \neg \neg a \leq \neg \Diamond a$, where the second inequality holds because $\Diamond$ is monotone. For the converse, we have that \[\neg \Diamond a \leq \Diamond a \to \Box 0 \leq \Box(a \to 0) = \Box \neg a.\] This shows that $\textup{(M5)}$ holds. For $\textup{(M6)}$, given $\textup{(M5)}$, we have that $\neg \Diamond\neg a = \Box\neg\neg a$. By $\textup{(l1)}$, $\textup{(M1)}$ and \textit{modus ponens}, we get:
\[\neg \Diamond \neg a \leq \Box \neg \neg a \land \Box(\neg\neg a \to a) \leq \Box(\neg\neg a \land (\neg\neg a \to a)) \leq \Box a ,\] as desired. Finally, $\textup{(FS3)}$ follows from $\textup{(FS2)}$.
\endproof

We will now prove a version of the $\textup{iEx}$-Embedding Theorem for $\textup{miEx}$ algebras. 
\begin{lemma} \label{modlma1}
    Let $L$ be a $miEx$ algebra and $O_L$ its quotient orthomodular algebra. Letting $\Box_Oa_O = [\Box a]_O$ and $\Diamond_Oa_O = [\Diamond a]_O$ for any $a \in L$, we have that the map $a \mapsto a_O$ is a surjective fundamental modal algebra homomorphism from $(L,\Box,\Diamond)$ to $(O_L,\Box_O,\Diamond_O)$. Moreover, $(O_L, \Box_O,\Diamond_O)$ is a $\mathrm{DOM}$ algebra.
\end{lemma}

\proof
By Lemma \ref{orthomorphlma}, we know that $O_L$ is an orthomodular algebra and that the map $a \mapsto a_O$ is a fundamental algebra homomorphism. We first show that it also preserves modalities. Let $a, b \in L$ be such that $a \sim b$. Then \[\neg \Diamond a = \Box \neg a=\Box \neg b = \neg \Diamond b,\] where the first and last equalities hold by $\textup{(M5)}$. This show that $\Diamond a \sim \Diamond b$. Moreover, we claim that $\neg \Box a = \neg\neg \Diamond \neg a$. The left-to-right direction follows from $\textup{(M6)}$ and the fact that $\neg$ is antitone, and, for the converse, by $\textup{(M4)}$, we have that $\neg \neg \Diamond \neg a \leq \neg\neg\neg \Box a \leq \neg \Box a$. By the claim, we now have that
\[\neg \Box a = \neg\neg\Diamond \neg a = \neg \neg\Diamond \neg b = \neg \Box b.\] Hence $\sim$ is a congruence relation with respect to $\Box$ and $\Diamond$ as well, which implies that $a \mapsto a_O$ is a fundamental modal algebra homomorphism from $(L,\Box,\Diamond)$ to $(O_L,\Box_O,\Diamond_O)$. 

Finally, we show that $(O_L,\Box_O,\Diamond_O)$ is a $\textup{DOM}$ algebra. Since it is a homomorphic image of $L$, we already know that $D_1$ and $D_2$ hold, so we only need to check $\textup{(M3)}$. This amounts to showing that $\neg\Diamond\neg a = \neg\neg\Box a$, which already follows from the claim above.
\endproof

This gives us the construction of $O_L$. For $I_L$, we first need to introduce the following restriction of Fischer Servi frames \cite{Ch23,FS84}:
\begin{definition}
    An $\mathrm{IFS}$ frame is a triple $(X, \leq, R)$ such that $(X, \leq)$ is a poset and $R$ a relation on $X$ satisfying the following conditions for any $x, y, z \in X$:
    \begin{itemize}
        \item[$\textup{(FC2)}$] $z \leq x$ and $xRy$ together imply that there is $y' \leq y$ such that $z R y'$;
        \item[$\textup{(IFC)}$] $xRy$ and $y \geq z$ together imply $y = z$.
    \end{itemize}
\end{definition}

Note that $\textup{(IFC)}$ is a strengthening of the usual $\textup{(FC1)}$ condition on Fischer Servi frames, which states that $x R y \geq z$ implies that there is $x' \leq x$ such that $x' R z$. As is to be expected, the downward closed subsets on an $\mathrm{IFS}$ frame form an $\mathrm{IFS}$ algebra. 

\begin{lemma} \label{ifsfrmlma}
    Let $(X, \leq, R)$ be an $\mathrm{IFS}$ frame. For any downset $U \sset X$, let $\Box_R U= \{x \in X \mid \forall y,z (x \geq y R z \to z \in U)\}$ and $\Diamond_R U = \{x \in X \mid \exists y: x R y \land y \in U\}$. Then $(Dn(X),\Box_R,\Diamond_R)$ is an $\mathrm{IFS}$ algebra.
\end{lemma}

\proof
It is well known that $\textup{(M1)}$, $\textup{(M2)}$, $\textup{(FS1)}$ and $\textup{(FS2)}$ hold on $(Dn(X),\Box_R,\Diamond_R)$ whenever $(X, \leq, R)$ satisfies conditions $\textup{(FC1)}$ and $\textup{(FC2)}$ (see, e.g., \cite{Ch23}). Hence we only need to check that $\textup{(l1)}$ holds. Fix $x \in X$ and a downset $U$. We must show that, whenever $x \geq y R z$, $z \in \neg \neg U \to U$. Since $y R z$, we know that $z$ is minimal, so we only need to verify that either $z \notin \neg \neg U$ or $z \in U$. Again, since $z$ is minimal, we have that, if $z \notin U$, then $z \in \neg U$, so that $z \notin \neg \neg U$, as desired.
\endproof

We can now lift the map $a \mapsto a_I$ to a fundamental modal homomorphism. 

\begin{lemma} \label{modlma2}
    Let $L$ be a $miEx$ algebra. Let $(P(L),\rset, R)$ be the set of all proper prime filters on $L$, ordered by reverse inclusion and endowed with the following relation $R$:
    \[P R Q \Leftrightarrow \big((\Box a \in P \Rightarrow a \in Q) \land (a \in Q \Rightarrow \Diamond a \in P)\big).\] Then $(Dn(P(L)),\Box_R, \Diamond_R)$ is an $\mathrm{IFS}$ algebra. Moreover, the map $a \mapsto a_I$ is a fundamental modal algebra homomorphism.
\end{lemma}

\proof
For the first part, by Lemma \ref{ifsfrmlma}, it is enough to show that $(P(L), \rset, R)$ satisfies $\textup{(FC2)}$ and $\textup{(IFC)}$. Let $P,P'$ and $Q$ be prime filters on $L$ such that $P' \rset P$ and $P R Q$. Because $P'$ is a prime filter and $\textup{(M1)}$, $\textup{(M2)}$ hold on $L$, note first that $F=\{a \mid \Box a \in P'\}$ is a filter and $I= \{b \mid \Diamond b \notin P'\}$ is an ideal. Let $F'$ be the filter generated by $Q \cup F$. We claim that $F' \cap I = \emptyset$. If this is true, then, since $L$ is an iEx-algebra, we can apply \ref{restPFT} and obtain a prime filter $Q'$ extending $F'$ and disjoint from $I$. By construction, for such a filter $Q'$ we have that $Q' \rset Q$ and $P'RQ'$, as desired. For the proof of the claim, suppose, towards a contradiction, that we have $c \in Q$, $a \in F$ and $b \in I$ such that $c \land a \leq b$. Then $1 \leq (c \land a) \to b$, which means that $(c \land a)\to b) \in Q$. We must then have $c \land ((c \land a) \to b) \in Q$, so $a \lor (a \to b) \in Q$ by $\mathrm{(iCl_1)}$. Since $Q$ is prime, either $a \in Q$ or $a\to b \in Q$. If $a \in Q$, we have that $(c \land a) \land ((c \land a) \to b)$, and hence $b \in Q$. Since $PRQ$, this implies that $\Diamond b \in P \sset P'$, contradicting our assumption. Hence we may assume that $a \notin Q$. We must then have $(a \to b) \in Q$, and hence $\Diamond(a \to b) \in P$. By $\textup{(FS3)}$, we get that $\Box a \in P$ or $\Box a \to \Diamond b \in P$. In the first case, it follows that $a \in Q$, which we have ruled out already. So $\Box a \to \Diamond b \in P \sset P'$. But since $\Box a \in P'$, this means that $\Diamond b \in P'$, a contradiction. This completes the proof of the claim, and thus that $\textup{(FC2)}$ holds.
For $\textup{(IFC)}$, we argue as follows. Suppose that $P, Q, Q'$ are prime filters such that $P R Q$ and $Q \sset Q'$. Letting $a \in L$, recall first that $1 \leq \neg\neg(a \lor \neg a)$ in any fundamental lattice. By $\textup{(l1)}$, we have that $\Box(\neg\neg (a \lor \neg a) \to (a \lor \neg a)) \in P$, so $\neg \neg (a \lor \neg a) \to (a \lor \neg a) \in Q$. Since we also have $\neg\neg (a\lor \neg a) \in Q$, it follows that $a \lor \neg a \in Q$. Because $Q$ is prime, it follows that $a \in Q$ or $\neg a \in Q$ for any $a \in L$. But this clearly means that $Q \sset Q'$ implies $Q' = Q$. Hence $\textup{(IFC)}$ holds as well.

This shows that $(Dn(P(L)),\Box_R, \Diamond_R)$ is an $\mathrm{IFS}$ algebra. For the second part, we know already that the map $a \mapsto a_I$ is a fundamental algebra homomorphism, but we must show that this map also preserves the modal operators. This amounts to showing the following for any $a \in L$ and any $P \in P(L)$:
\begin{enumerate}
    \item $\Diamond a \in P$ if and only if there is $Q \in P(L)$ such that $a \in Q$ and $P R Q$;
    \item $\Box a \in P$ if and only if, for any $P', Q \in P(L)$ such that $P\sset P'$ and $P' R Q$, $a \in Q$.
\end{enumerate}

We start from the first item. The right-to-left direction is clear by the definition of the relation $R$. For the converse, we assume that $\Diamond a \in P$, and we first claim that the filter $F$ generated by the set $\{a\} \cup \{b \mid \Box b \in P\}$ is disjoint from the ideal $I = \{d \mid \Diamond d \notin P\}$. Assume towards a contradiction that we have $a \land b \leq d$ for some $d \in I$ and $b$ such that $\Box b \in P$. By $\textup{(M2)}$, we get that $\Diamond (a \land b) \leq \Diamond d$. Moreover, by $\textup{(D1)}$, we have that $\Diamond a \land \Box b \leq \Diamond (a \land b) \leq \Diamond d$. Since $\Diamond a$ and $\Box b$ belong to $P$ by assumption, this means that $\Diamond d \in P$, a contradiction. Now let $Q$ be a filter extending $F$ that is maximal among those disjoint from $I$. We claim that $Q$ is prime. If true, this clearly completes the proof of the first item above. For the proof of the claim, suppose that $c_1, c_2 \notin Q$. By a standard argument, there must be $d \in I$ and $b$ such that $\Box b \in P$ such that:
$(a \land b) \land c_1 \leq d$ and $(a \land b) \land c_2 \leq d$. By standard reasoning using $\textup{(M1)$, $\textup{(M2)}}$ and Lemma \ref{LemmaInequalityEx}, this implies that: 
\begin{itemize}
    \item $1 \leq \Box((((a \land b) \land c_1) \lor ((a \land b) \land c_2)) \to d)$;
    \item $\Diamond((a \land b) \land c_1)\leq \Diamond d$;
    \item $\Diamond((a \land b) \land c_2)\leq \Diamond d$.
\end{itemize}
Together with $\textup{(MCl)}$, the first inequality implies that
\[\Box (((a \land b ) \land (c_1 \lor c_2)) \to d) \lor \Diamond ((a\land b) \land c_1) \lor \Diamond((a \land b) \land c_2)\] is in $P$. Since $P$ is prime, one of $\Box (((a \land b ) \land (c_1 \lor c_2)) \to d)$, $\Diamond ((a\land b) \land c_1)$ or $\Diamond ((a\land b) \land c_2)$ must belong to $P$. In the first case, this implies that $((a \land b) \land (c_1 \lor c_2)) \to d \in Q$. Since $a \land b \in Q$ and $d \notin Q$, this means that $c_1 \lor c_2 \notin Q$. In the second and third case, the two inequalities above imply that $\Diamond d  \in P$, a contradiction. Therefore $c_1 \lor c_2 \notin Q$, and $Q$ is prime, as desired. 

Finally, we show that item $2$ above also holds. The left-to-right direction clearly follows from the definition of $R$. For the converse, suppose that $\Box a \notin P$. By $\textup{(M6)}$, this means that $\neg \Diamond \neg a \notin P$. Since $L$ is an $\textup{Ex}$-algebra, we can apply Lemma \ref{LemmaExEmbeddingiCl}, and obtain a prime filter $P'$ extending $P$ and containing $\Diamond \neg a$. By the proof of item $1$ above, $\Diamond \neg a \in P'$ implies that there is a prime filter $Q$ such that $\neg a \in Q$ and $P' R Q$. Clearly, $a \notin Q$, which is what we needed. This completes the proof.
\endproof

Putting Lemmas \ref{modlma1} and \ref{modlma2} together, we can now prove a modal version of the $\textup{iEx}$-Embedding Theorem.
\begin{theorem}
    Let $L$ be a $miEx$ algebra. Then there are a $\mathrm{DOM}$ algebra $O_L$, an $\mathrm{IFS}$ algebra $I_L$ and a subdirect fundamental modal algebra embedding $e: L \to O_L \times I_L$. 
\end{theorem}

\proof
Fix an $miEx$ algebra $L$. We let $O_L$ be the quotient ortholattice of $L$, and $I_L$ be the fundamental modal subalgebra of $(Dn(P(L)),\Box_R, \Diamond_R)$ generated by sets of the form $a_I$ for $a \in L$. By Lemmas \ref{modlma1} and \ref{modlma2}, $O_L$ is a $\mathrm{DOM}$ algebra and $I_L$ is an $\mathrm{IFS}$ algebra, and both are fundamental modal homomorphic images of $L$. By the same reasoning as the proof of the iEx-Embedding Theorem, the map $e: L \to O_L \times I_L$ given by $e(a) = (a_O,a_I)$ is therefore a fundamental modal algebra subdirect embedding.
\endproof

Finally, by the same reasoning as in the proof of Theorem \ref{TheoremEquivExalgebras}, we now get the following:
\begin{corollary}
    Let $\mathsf{DOM}$ be the modal logic of $\mathrm{DOM}$ algebras and $\mathsf{IFS}$ the modal logic of $\mathrm{IFS}$ algebras. Then $\mathsf{DOM} \cap \mathsf{IFS}$ is the logic of $\textup{miEx}$ algebras.
\end{corollary}

\section{Conclusion}

We conclude with some open problems about constructive quantum logics motivated by the results presented here:
\begin{description}
    \item[Question 1] What is the complexity of $\textup{iEx}$-logic?
    \item[Question 2] Are all the axioms listed in Theorem \ref{altaxthm} independent over Holliday's preconditional logic?
    \item[Question 3] Is there a finite axiomatization of $\textup{iEx}$-logic with axioms that contain at most three variables?
    \item[Question 4] What are the joint validities of basic modal orthologic and the implication-free fragment of $\mathsf{IK}$ or $\mathsf{CK}$?
    \item[Question 5] What are the joint validities of basic modal orthomodular logic and $\mathsf{IK}$ or $\mathsf{CK}$?
\end{description}

\subsection*{Acknowledgements}
The authors thank the referees for helpful comments. The first author would also like to acknowledge support of the Austrian Science Foundation (FWF) through grants STA-139 and ESP-3.

\nocite{*}
\bibliographystyle{eptcs}
\bibliography{bibliography}

\end{document}